\documentclass[a4paper,12pt]{amsart}
\usepackage{amssymb,amsthm,amsmath,psfrag,hyperref}
\usepackage{mathrsfs}
\usepackage[utf8]{inputenc}
\usepackage[T1]{fontenc}
\usepackage{color}

\definecolor{refkey}{gray}{.75}
\definecolor{labelkey}{gray}{.50}
\newtheorem{theorem}{Theorem}[section]
\newtheorem{lemma}[theorem]{Lemma}
\newtheorem{proposition}[theorem]{Proposition}

\theoremstyle{definition}

\def\FF{\mathscr{F}}
\def\MM{\mathscr{M}}
\def\Re{{\rm Re}}
\def\Im{{\rm Im}}

\theoremstyle{remark}

\numberwithin{equation}{section}

\let\epsilon=\varepsilon
\def\ep{{\varepsilon}}
\def\R{\mathbb R}
\def\E{{\rm e}^{t(J_0(\xi)-1)}}
\def\dirac{\delta_{0}}
\let\phi=\varphi
\newtheorem{ex}[theorem]{\textbf{Example}}
\newtheorem{assumption}[theorem]{\textbf{Assumption}}
\newtheorem{rem}[theorem]{\textbf{Remark}}

\title[Blow up in a half-space]{Blow-up phenomena for
positive solutions of  \\ 
semilinear diffusion equations in a half-space: \\ 
the influence of the dispersion kernel}
\date{}

\begin{document}

\maketitle

\begin{center}{
	\large\bf Matthieu Alfaro\footnote{
		IMAG, Univ. Montpellier, CNRS, Montpellier, France \& BioSP, INRA, 84914, Avignon, France.
		e-mail: {\tt matthieu.alfaro@umontpellier.fr}
	} 
	and 
	Otared Kavian\footnote{
		Université Paris-Saclay (site de Versailles);
		45 avenue des Etats-Unis; 
		78035 Versailles cedex; France.
		e-mail: {\tt otared.kavian@uvsq.fr} 
	}
	}\\

\end{center}



\vspace{10pt}

\begin{abstract} We consider the semilinear diffusion equation $\partial_{t} u = Au +  |u|^{\alpha}u$ in the half-space $\R_{+}^N := \R^{N-1}\times(0,+\infty)$, where $A$ is a linear diffusion operator, which may be the classical Laplace operator, or a fractional Laplace operator, or an appropriate non regularizing nonlocal operator. The equation is supplemented with an initial data $u(0,x) = u_{0}(x)$ which is nonnegative in the half-space $\R_{+}^N$, and the Dirichlet boundary condition $u(t,x',0) = 0$ for $x'\in \R^{N-1}$. 

We prove that if the symbol of the operator $A$ is of order $a|\xi|^\beta$ near the origin $\xi = 0$, for some $\beta \in (0,2]$, then any positive solution of the semilinear diffusion equation blows up in finite time whenever $0 < \alpha \leq \beta/(N+1)$. On the other hand, we prove existence of positive global solutions of the semilinear diffusion equation in a half-space when $\alpha > \beta/(N+1)$. Notice that in the case of the half-space, the exponent $\beta/(N+1)$ is smaller than the so-called Fujita exponent $\beta/N$ in $\R^N$.  

As a consequence we can also solve the blow-up issue for solutions of the above mentioned semilinear diffusion equation in the whole of $\R^N$, which are odd in the $x_{N}$ direction (and thus sign changing).

\medskip
\noindent{\bf Key Words:} sign changing solutions, half-space, blow-up solutions, global solutions, Fujita exponent, nonlocal diffusion, dispersal tails.

\medskip
\noindent{\bf AMS Subject Classifications:} 35B40 (Asymptotic behavior of solutions), 35B33 (Critical exponent), 45K05 (Integro partial diff eq), 47G20 (Integro diff oper).

\end{abstract}

\tableofcontents

\section{Introduction}\label{s:intro}

Let $N \geq 1$ be an integer. A generic point $x\in \R^{N}$ will be written $x=(x',x_{N})$ with $x'\in \R^{N-1}$ and $x_{N}\in \R$. We denote the upper half-space $\R^{N} _{+}:=\R^{N-1}\times (0,+\infty)$, its boundary being $\partial\R^N_{+} = \R^{N-1}\times\big\{0\big\}$.  
 For given real numbers $\alpha > 0$ and $T > 0$, in this work we are interested in the blow-up of positive solutions $u$ to the semilinear equation
\begin{equation}
\label{eq:1}
\begin{cases}
\partial_t u = Au +  |u|^{\alpha}u  &\mbox{in }\, (0,T)\times \R_{+}^{N}, \\
u(0,x) = u_{0}(x) \geq 0 &\mbox{for }\, x \in \R_{+}^{N-1}, \\
u(t,x',0) = 0 &\mbox{for }\, (t,x')\in [0,T]\times \R^{N-1},
\end{cases}
\end{equation}
where $A$ is a linear diffusion operator, which may be the classical Laplace operator, or a fractional Laplace operator, or an appropriate non regularizing nonlocal operator.

\medskip

If $A := \Delta$, the Laplace operator on $\R^N$, since the seminal work of H. Fujita \cite{Fuj-66} (later completed by  K. Hayakawa \cite{Hay-73}, K. Kobayashi, T. Sirao, and H. Tanaka \cite{Kob-Sir-Tun-77}, F. B. Weissler \cite{Wei-81}), it is known that when $u \not\equiv 0$ is the solution of
\begin{equation}\label{eq:Heat-space}
\partial_t u = \Delta u +  u^{\alpha + 1} \quad \mbox{in }\, (0,T)\times \R^{N}, \quad
u(0,x) = u_{0}(x) \geq 0 \quad\mbox{for }\, x \in \R^{N} ,
\end{equation}
and if $0 < \alpha \leq  2/N$ then one must have $ T < +\infty$, whereas if $\alpha>2/N$ one may have $T=+\infty$ for some \lq\lq small'' nontrivial initial data. For problem \eqref{eq:Heat-space}, the exponent $p_F:=1+2/N$ is then referred to as the {\it Fujita exponent}.  In \cite{Kav-87} the second author shows that in the case of the half-space, corresponding to  $A = \Delta$ in equation \eqref{eq:1},  the Fujita exponent is $p_{F} := 1 + 2/(N+1)$. As a matter of fact, if the half-space $\R_{+}^N$ is replaced with a domain $\Omega \neq \R^N$ which is a cone with vertex at the origin, there exists an exponent $\alpha_{*}(\Omega) < 2/N$, such that any positive solution blows up in finite time when $0 < \alpha \leq \alpha_{*}(\Omega)$. On the other hand one can show that there exist positive global solutions when $\alpha > \alpha_{*}(\Omega)$, and $(N-2)\alpha < 4$ (so that when $N \geq 3$ one has $\alpha + 2 < 2^* := 2N/(N-2)$, the critical Sobolev exponent).

When the operator $A$ is given by 
\begin{equation}\label{eq:A-kernel}
(Au)(x) := \int_{\R^N} J(x-y)u(y)dy - u(x),
\end{equation}
where $J \in L^1(\R^N)$ is a nonnegative function such that for some constants $a > 0$ and $\beta \in (0,2]$ one has
$$
{\widehat J}(\xi) = 1 - a|\xi|^\beta + o(|\xi|^\beta)\quad\mbox{as }\, \xi \to 0,
$$
then in \cite{Alf-fujita} the first author shows that the Fujita exponent for the problem
\begin{equation}\label{eq:space}
\partial_t u = Au +  u^{\alpha + 1} \quad \mbox{in }\, (0,T)\times \R^{N}, \quad
u(0,x) = u_{0}(x) \geq 0 \quad\mbox{for }\, x \in \R^{N} ,
\end{equation}
is $p_F=1+\beta/N$. For the fractional Laplacian, the corresponding result is proved by S. Sugitani \cite{Sug-75}: when $A = -(-\Delta)^{\beta/2}$ for $\beta \in (0,2)$, the Fujita exponent is $p_F=1+\beta /N$. 

One can interpret these results as a consequence of the fact that, for such weakly nonlinear diffusion equations on $\R^N$, the zero solution is highly unstable, no matter the norm one may consider, as long as one is interested in the global existence of solutions. This is in contrast with the analogous semilinear diffusion equation on a bounded domain $\Omega \subset \R^N$, 
\begin{equation}
\label{eq:1-Bounded}
\begin{cases}
\partial_t u = \Delta u +   |u|^{\alpha}u  &\mbox{in }\, (0,T)\times \Omega, \\
u(0,x) = u_{0}(x)  &\mbox{for }\, x \in \Omega, \\
u(t,\sigma) = 0  &\mbox{for }\, (t,\sigma)\in [0,T]\times \partial\Omega.
\end{cases}
\end{equation}
Indeed for such equations, an initial data $u_{0}$ which is small enough in an appropriate norm, gives rise to a global solution, no matter the sign of $u_{0}$ (see for instance Th. Cazenave \& A. Haraux \cite[Chapter 5]{CazenaveThHaraux}).

As a matter of fact, on a bounded domain, a situation in which all nonnegative solutions blow up in finite time is the following equation which a is simple perturbation of \eqref{eq:1-Bounded}. To see this, let us denote by $\lambda_{1} > 0$ the first eigenvalue of the operator $(L,D(L))$ defined by
$$Lu := -\Delta u, \qquad \mbox{for }\, 
u \in D(L) := \left\{v \in H^1_{0}(\Omega) \; ; \; \Delta v \in L^2(\Omega)\right\}.$$
Then it is easy to see that for any $\lambda \geq \lambda_{1}$ the solution of
\begin{equation}
\label{eq:Bounded-lambda}
\begin{cases}
\partial_t u = \Delta u + \lambda u +   u^{\alpha + 1} & \mbox{in }\, (0,T)\times \Omega, \\
u(0,x) = u_{0}(x) \geq 0 & \mbox{for }\, x \in \Omega, \\
u(t,\sigma) = 0 & \mbox{for }\, (t,\sigma)\in [0,T]\times \partial\Omega,
\end{cases}
\end{equation}
blows up in finite time when $u_{0} \not\equiv 0$ and $\alpha > 0$. Indeed, denote by $\phi_{1} > 0$ the eigenfunction associated to $\lambda_{1}$, that is
$$-\Delta \phi_{1} = \lambda_{1}\phi_{1}, \qquad \phi_{1} \in H^1_{0}(\Omega), \qquad
\int_{\Omega}\phi_{1}(x)dx = 1,$$
and multiply the first equation of \eqref{eq:Bounded-lambda} by $\phi_{1}$. On sees that if we set 
$$M(t) := \int_{\Omega}u(t,x)\phi_{1}(x)dx,$$ 
then since $\int_{\Omega}\phi_{1} \Delta u dx = \int_{\Omega} u\Delta \phi_{1}dx = -\lambda_{1}M(t)$, we have 
$$M'(t) = \int_{\Omega}\phi_{1} \Delta u dx + \lambda M(t) + \int_{\Omega}u^{\alpha + 1}\phi_{1}dx \geq (\lambda - \lambda_{1})M(t) + M(t)^{\alpha + 1}.$$
Where we have used the Jensen inequality 
$$\int_{\Omega}u^{\alpha + 1}(t,x)\phi_{1}(x)dx \geq \left(\int_{\Omega}u(t,x)\phi_{1}(x)dx\right)^{\alpha + 1}.$$
Finally, when $\lambda \geq \lambda_{1}$, we see $M'(t) \geq M(t)^{1+\alpha}$, and this implies that the solution $u(t)$ cannot exist for all times $t > 0$.

\section{Main results and comments}\label{s:main}

In this paper we are interested in the case in which the operator $A$ is a rather general diffusion operator (including the above examples)  and the domain on which the diffusion equation is considered is a half-space. Depending on the linear operator $A$, equation \eqref{eq:1} can be a partial differential equation or an integro-differential equation.

\subsection{Assumptions and main results}

The class of initial data considered is those satisfying the following assumption.

\begin{assumption}[Initial data] \label{ass:initial}
 The initial data $u_0 \in L^{\infty}(\R_{+}^{N})\cap L^{1}(\R_{+}^{N})$ satisfies
\begin{equation}\label{eq:Ini-data}
u_{0}\not\equiv0, \quad
\forall x \in \R_{+}^N, \quad u_0(x)\geq 0, \quad 
m_{1}=m_1(u_0) := \int_{\R_{+}^N} x_{N}u_{0}(x)dx < +\infty.
\end{equation}
\end{assumption}

In order to setup the types of linear operators which can be treated by our approach, we start by introducing a function $J_0$ as follows.

\begin{assumption}[Function $J_0$]\label{ass:J0}
Let $J_{0}$ be a function satisfying
\begin{equation}\label{eq:J0}
J_{0}\in C(\R^N;\R), \qquad
J_{0}(\xi) = 1 - a|\xi|^{\beta} + o(|\xi|^{\beta}), \quad \text{ as } \xi \to 0,
\end{equation}
where 
\begin{equation}\label{eq:a-beta}
a > 0, \qquad 0 < \beta \leq 2,
\end{equation}
and 
\begin{equation}\label{eq:J0-infini}
\forall r > 0, \qquad \sup_{|\xi| \geq r} J_{0}(\xi) < 1.
\end{equation}
\end{assumption}

Equipped with such a function $J_0$, we now introduce a kernel $G$ as follows.

\begin{assumption}[The kernel $G$]\label{ass:kernel-G}
We denote by $\MM_{b}(\R^N)$ the space of bounded measures on $\R^N$. We shall say that a measure $\mu \in \MM_{b}(\R^N)$ is even with respect to the $x_N$ variable, (resp. with respect to $x'$ variable), when $\langle \mu, \phi \rangle = 0$ for any function $\phi \in C_{b}(\R^N)$ such that $\phi(x',-x_{N}) = -\phi(x',x_{N})$,  (resp. such that $\phi(-x',x_{N}) = - \phi(x',x_{N})$), for all $x = (x',x_{N})\in \R^N$. Also, as the definition of the Fourier transform on $\R^N$ we shall take
$$(\FF(f))(\xi) = \widehat f(\xi) := \int _{\R^N}{\rm e}^{-{\rm i}x\cdot \xi }f(x)dx, $$
for $f \in L^{1}(\R^{N})$.

We consider a kernel $G$ such that
\begin{equation}\label{eq:G-L1}
G \in C((0,+\infty); \MM_{b}(\R^N)), \quad G(t) \geq 0, \quad
\quad \forall t >0, \quad G(t)(\R^N) = 1,
\end{equation}
and for all $t > 0$, 
\begin{equation}
\label{parite-kernel}
G(t)\,\text{ is even with respect to }\, x_{N}, \qquad
G(t)\,\text{ is even with respect to }\, x'.
\end{equation}
Moreover we assume that there exists a function $J_{0}$ satisfying Assumption \ref{ass:J0}, namely conditions \eqref{eq:J0}--\eqref{eq:J0-infini}, such that the Fourier transform of $G$ is written
\begin{equation}
\label{fourier-kernel}
\FF(G(t)) (\xi) = {\rm e}^{t(J_0 (\xi)-1)}.
\end{equation}
\end{assumption}

Equipped with such a kernel $G$, we now state our assumption on the linear operator $A$.

\begin{assumption}[The linear diffusion equation]\label{ass:diffusion}
For a kernel $G$ satisfying Assumption \ref{ass:kernel-G}, and for any initial data $v_{0} \in \MM_{b}(\R^N)$, we assume that the unique solution of the Cauchy problem
\begin{equation}
\label{cauchy-diffusion}
\partial_{t} v = Av \quad \text{ in } (0,+\infty)\times \R^{N}, \quad v(0,\cdot) = v_{0},
\end{equation}
is given by the convolution
\begin{equation}
\label{convolution}
v(t) = G(t)*v_{0}.
\end{equation}
Note that condition \eqref{parite-kernel} implies in particular that $G(t)$ is even with respect to $x$, thus the semigroup is self-adjoint and so is its infinitesimal generator $A$.
\end{assumption}

For $u_0\in L^{\infty}(\R^{N}_+)\cap L^{1}(\R^{N}_+)$, the above assumptions enable us to give a sense to both the linear (possibly nonlocal) Cauchy problem in the half-space
\begin{equation}\label{eq:Cauchy-Half-Space-linear}
\begin{cases}
\partial_{t} u = Au&\text{ in } (0,+\infty)\times \R ^{N}_+\\
u(0,\cdot) = u_{0} &\text{ in } \R ^{N}_+\\
u(t,x',0) = 0 &\text{ on } \partial \R ^{N}_+,
\end{cases}
\end{equation}
and the semilinear (possibly nonlocal) Cauchy problem \eqref{eq:1} in the half-space $\R^N_+$. This will be clarified  in Section \ref{s:cauchy}. The last required assumption is the following.
\begin{assumption}[Comparison principle]\label{ass:comparison}
The Cauchy problems \eqref{eq:Cauchy-Half-Space-linear} and \eqref{eq:1} in the half-space enjoy the comparison principle. In particular, for a nonnegative initial data $u_0$ as in Assumption \ref{ass:initial}, the solution $u = u(t,x)$ to \eqref{eq:Cauchy-Half-Space-linear}, or to \eqref{eq:1}, remains nonnegative as long as it exists. 
\end{assumption}

Notice that the above assumption is automatically satisfied in many situations, see Remark \ref{rem:nonnegativity}.

\medskip

Under the above set of  assumptions, our main result is the identification of the Fujita exponent $p_F := 1 + \beta/(N+1)$ for equation \eqref{eq:1}. 
\begin{theorem}[Systematic blow-up]
\label{th:systematic}
Let Assumptions \ref{ass:initial}, \ref{ass:J0}, \ref{ass:kernel-G}, \ref{ass:diffusion} and \ref{ass:comparison} hold. 
Denoting
$$p_{F} := 1 + \frac{\beta}{N+1}$$
assume that $0 < \alpha  < \beta/(N+1) = p_F - 1 $. Then  any nontrivial solution $u$ to the Cauchy problem \eqref{eq:1} blows up in finite time.
\end{theorem}

For the critical case when $\alpha = \beta/(N+1)$, the proof of systematic blow-up is more delicate but we can prove the result under some slight restrictions on the linear operator $A$.

\begin{theorem}[Systematic blow-up when $\alpha = \beta/(N+1)$]
\label{th:systematic-1} Let Assumptions \ref{ass:initial}, \ref{ass:J0}, \ref{ass:kernel-G}, and \ref{ass:diffusion}  hold and let $\alpha = \beta/(N+1)$.
Moreover assume that either $A = \Delta$ the Laplacian operator, or $A=-(-\Delta)^{\beta/2}$  the fractional Laplacian operator ($0<\beta<2$), or $Au = J*u - u$ for a nonnegative $J\in L^{1}(\R^{N})\cap L^{\infty}(\R^{N})$ such that 
$$J(x',-x_{N}) = J(x',x_{N}), \qquad J(-x',x_{N}) = J(x',x_{N}), $$
and moreover 
\begin{equation}\label{eq:Cond-J-cri}
\text{ for all } z'\in \R^{N-1}, z_N\mapsto J(z',z_N) \text{ is nonincreasing on $(0,+\infty)$.}
\end{equation}
Then  any nontrivial solution $u$ to the Cauchy problem \eqref{eq:1} blows up in finite time.
\end{theorem}
Notice that the above restrictions on $A$ insure in particular that Assumption \ref{ass:comparison} is automatically satisfied , see Remark \ref{rem:nonnegativity}, and therefore there is no need to require Assumption \ref{ass:comparison} in the statement of Theorem \ref{th:systematic-1}.

For a function $f \in L^1(\R_{+}^N)$ we define its odd extension ${\widetilde f}$ by setting
$${\widetilde f}(x) := f(x)\quad\text{if }\, x_{N} > 0, \qquad
{\widetilde f}(x) := -f(x)\quad\text{if }\, x_{N} < 0.$$
Then we have the following result regarding the existence of global nonnegative solutions to \eqref{eq:1} when $\alpha > \beta/(N+1)$.
\begin{theorem}[Possible global existence and extinction]\label{th:global} With the assumptions of Theorem \ref{th:systematic} let $\alpha > p_F - 1 = \beta/(N+1)$.  Then there exists $\ep_{*} > 0$ such that, for any initial data $u_{0}$ satisfying 
$$
m_1(u_0)+\|{\mathscr{F}({\widetilde u}_{0})}\|_{L^{1}} < \ep_{*},
$$
the solution to the Cauchy problem \eqref{eq:1} is global in time and  satisfies, for some $C=C(u_0)>0$,
\begin{equation}
\label{extinction}
\forall t \geq 0, \qquad \|u(t,\cdot)\|_{L^\infty}\leq C (1+t)^{-(N+1)/\beta}.
\end{equation}
\end{theorem}

\subsection{Examples and comments}

Here are a few examples of diffusion operators $A$, kernels $G$ and functions $J_{0}$ which satisfy Assumption \ref{ass:J0} through Assumption \ref{ass:diffusion}.

\begin{ex}[Regularizing operators]\label{ex:reg} Consider $Au := \Delta u$ and 
$$G(t,x) := (4\pi t)^{-N/2} {\rm e}^{-|x|^{2}/4t},$$
that is is the classical heat kernel. Then its Fourier transform is given by
$$\FF(G(t,\cdot))(\xi) = {\rm e}^{t(J_{0}(\xi) - 1)}, \qquad\text{with}\;
J_{0}(\xi) := 1 - |\xi|^2.$$
If $Au := -(-\Delta)^{\beta /2}u$, for some $\beta \in (0,2)$ then $G(t,x)$ is the fractional heat kernel whose Fourier transform is given by
$$\FF(G(t,\cdot))(\xi) = {\rm e}^{t(J_{0}(\xi) - 1)}, \qquad\text{with}\;
J_{0}(\xi) := 1 - |\xi|^\beta.$$
One checks that all the conditions involved in the Assumption \ref{ass:J0} through Assumption \ref{ass:diffusion} are satisfied.
\qed
\end{ex}

\begin{ex}[Non regularizing operators]\label{ex:nonreg} Let $J \in L^1(\R^N) \cap L^\infty(\R^N)$ be a nonnegative function and define the non local diffusion operator $$Au : = J*u - u.$$
Assume that  $J$ is even in the $x_N$ variable, as well as  in the $x'$ variable, and such that its Fourier transform ${\widehat J}$ verifies
\begin{equation}
\label{eq:J0-bis}
{\widehat J}(\xi) = 1 - a |\xi|^\beta + o(|\xi|^\beta),\quad \text{ as $\xi \to 0$}.
\end{equation}
 This kind of nonlocal  operator is very relevant in population dynamics, in order to take into account long-distance dispersal events, see for instance  \cite{Med-Kot-03}, \cite{Cov-Dup-07}, \cite{Gar-11}, \cite{Alf-Cov-17}. 

In this case, the kernel $G$ is in fact the measure
\begin{equation}
\label{G-measure}
G(t,\cdot) := {\rm e}^{-t}\dirac  
+ {\rm e}^{-t}\sum_{k=1}^{+\infty}\frac{t^k}{k!}J^{*(k)},
\end{equation}
where $\dirac$ denotes the Dirac mass at the origin and $J^{*(k)}$ is defined by setting $$J^{*(1)} := J, \quad\text{and}\quad J^{*(k+1)} := J*J^{*(k)}\quad\text{for }\, k \geq 1.$$ 
Then Assumption \ref{ass:diffusion} is satisfied with  $J_0(\xi) := {\widehat J}(\xi)$.
In order to verify that $J_0 := \widehat J$ satisfies \eqref{eq:J0-infini}, for a given $r > 0$, since $\widehat J \in C_{0}(\R^N)$ we may fix $R > r$ such that $\widehat J(\xi) \leq 1/2$ for $|\xi| \geq R$. Now, since $J(x) > 0$ a.e on some ball $B(x_0,\ep)$, we have, for $r \leq |\xi| \leq R$,
$$
\widehat J(\xi) - 1 = \int_{\R^N}J(x)(\cos(x\cdot \xi) - 1)dx \leq \int_{\vert x-x_0\vert <\ep}J(x)(\cos(x\cdot \xi) - 1)dx < 0,
$$
since $\cos(x\cdot \xi) - 1 \leq 0$ on $\R$, and $\cos(x\cdot \xi) - 1  \not\equiv 0$ on $B(x_0,\ep)$. Hence, $\widehat J$ being continuous, we have
$$\delta_{0} := \max_{r \leq |\xi| \leq R} \widehat J(\xi) < 1,$$
and thus \eqref{eq:J0-infini} is satisfied with $\delta := \min(1/2,\delta_{0})$.

As a matter of fact, expansion \eqref{eq:J0-bis}  contains some information on the tails of the convolution kernel $J$. Indeed, as soon as  $J$ has a finite second momentum, namely 
$$m_2 := \int_{\R ^N}| x| ^2J(x)dx<+\infty,$$
the expansion \eqref{eq:J0} holds true with $\beta = 2$, as can be seen in R. Durrett \cite[Chapter 2, subsection 2.3.c, (3.8) Theorem]{Dur-96} among others.  On the other hand, when $m_2 = +\infty$ then more general expansions are possible, with in particular $0<\beta<2$, see the discussion in  \cite[Section 2]{Alf-fujita}. 

For a detailed analysis of the diffusion equation associated with such non regularizing operators, we refer to \cite{Cha-Cha-Ros-06}. 
\qed
\end{ex}

We now comment on some of the assumptions in our main results.

\begin{rem}[On assumption \eqref{eq:J0-infini}] As soon as there is $t_0 > 0$ such that $G(t_0)\in L^1(\R^{N})$, then it follows that ${\widehat G}(t_{0},\xi) \to 0$ as $|\xi|\to+\infty$, so that from \eqref{fourier-kernel} we deduce $J_0(\xi)\to -\infty$ as $| \xi| \to +\infty$, while ${\widehat G}(t_{0},\xi) < 1$ for $|\xi| \geq r > 0$.  Thus condition \eqref{eq:J0-infini} is superfluous. This happens in particular for the regularizing operators of Example \ref{ex:reg}. 

Assumption \eqref{eq:J0-infini} is also superfluous as soon as $J_0(\xi)$ is the Fourier transform of an integrable function, which is the case of the non regularizing operators of Example \ref{ex:nonreg}.

However, in general condition \eqref{eq:J0-infini} is necessary for the proof of our result, as it cannot be deduced from our other assumptions on the diffusion operator $A$: for instance consider the case $N = 1$ and the measure 
$$\mu := \frac{1}{2}\left(\delta_{-1} + \delta_{+1}\right),$$ 
where $\delta_{a}$ denotes the Dirac measure at $a \in \R$. Then if we define the operator $A$ by setting $A v := \mu*v - v$ for $v \in \MM_{b}(\R)$, noting that ${\widehat \mu}(\xi) = \cos(\xi)$, it is easy to see that the Fourier transform of the kernel $G(t)$ is given by
$$\FF(G(t))(\xi) = {\rm e}^{t(\cos(\xi)-1)},$$
so that $J_{0}(\xi) = \cos(\xi)$ does not satisfy \eqref{eq:J0-infini}. Therefore we cannot treat such a diffusion operator with the method developed here.\qed
\end{rem}

\begin{rem}[On the comparison principle assumption]\label{rem:nonnegativity} Assumption \ref{ass:comparison}  is automatically fulfilled in the case of the Laplacian and the fractional Laplacian (see Example \ref{ex:reg}), and of the convolution operator as stated in Example \ref{ex:nonreg} provided condition \eqref{eq:Cond-J-cri} is satisfied. We refer to Remark \ref{rem:nouvelle} for details. \qed
\end{rem}

\subsection{Organization of the paper} The remainder of the paper is organized as follows. In Section \ref{s:basic}, we prove preliminary important estimates on the linear diffusion equation \eqref{cauchy-diffusion}. Linear and semilinear Cauchy problems in the half-space $\R^N_+$ are discussed in Section \ref{s:cauchy}. We devote Section \ref{s:blowup} to the proof of the systematic blow-up when $0< \alpha < p_F - 1$, as stated in Theorem \ref{th:systematic}, whereas the critical case $\alpha=p_F-1$, as stated in Theorem \ref{th:systematic-1}, is considered in Section \ref{s:critical}. Finally, in Section \ref{s:extinction} we prove Theorem \ref{th:global}, that is the existence of global solutions when $\alpha > p_F - 1$.


\section{Notations and linear preliminary results}\label{s:basic}

Our conventions on the Fourier transform are the following. If $f \in \mathscr{S}(\R^N)$, the space of L. Schwartz rapidly decaying functions, or $f\in L^1(\R ^N)$, we define its Fourier transform $\FF(f) = \widehat f$ and its inverse Fourier transform $\FF^{-1}(f)$ by 
$$(\FF(f))(\xi) = \widehat f(\xi) := \int _{\R^N}{\rm e}^{-{\rm i}x\cdot \xi }f(x)dx, $$
and 
$$
(\FF^{-1}(f))(x) := (2\pi)^{-N} \int _{\R^N} {\rm e}^{{\rm i}x\cdot \xi }f(\xi)d\xi ,
$$
which means that $\FF^{-1} = (2\pi)^{-N}\FF^*$, where $\FF^*$ is the adjoint of $\FF$ in $L^2({\R}^N)$, after extending the definition of the Fourier transform to $L^2(\R^N)$. With this definition, we have, for $f,g\in L^1(\R^N)$,
$$
\FF(f*g) = \widehat{f*g} = \FF(f)\FF(g),
$$
and also the Plancherel formula
\begin{equation}\label{eq:Plancherel}
\int_{\R ^N} f(x)\overline g(x)dx =  (2\pi)^{-N} \int_{\R ^N} \widehat f(\xi)\overline{\widehat g}(\xi)d\xi,
\end{equation}
for $f$, $g \in L^2(\R ^N)$, which is equivalent to $\FF\FF^* = \FF^*\FF = (2\pi)^N I$ on $L^2(\R^N)$.

\subsection{$L^{\infty}$ and pointwise estimates for the linear diffusion equation}

This subsection is devoted to $L^{\infty}$ and pointwise estimates on solutions $v(t,x)$ of the linear diffusion equation \eqref{cauchy-diffusion} starting from an initial data $v_0$ which is odd in the $x_{N}$ variable. We start with  estimates of the Fourier transform of the initial data for small frequencies $\xi$. Recall that we have set
$$
m_{1}(v_{0}) := \int_{\R^N_{+}} x_{N} v_{0}(x) dx .
$$

\begin{lemma}[Fourier transform of the initial data]\label{claim} Let $v_{0} \in L^1(\R^N)$ be such that 
\begin{equation}\label{eq:v0-odd}
v_{0}(x) \geq 0\;\text{ for }\; x_{N} > 0, \qquad v_{0}(x',-x_{N}) = - v_{0}(x',x_{N}), \qquad 
m_{1}(v_{0}) < +\infty.
\end{equation}
Then, as $\xi \to 0$, we have (recall that $\xi=(\xi',\xi _N)$)
$$\Re({\widehat{v_0}}(\xi)) = o(\xi_{N}) \quad \text{ and } \quad  
\Im({\widehat{v_0}}(\xi)) = -\xi_N(2m_1(v_0)+o(1)),
$$
and thus $| \widehat{v_0}(\xi)| = | \xi _N| (2m_1(v_0)+o(1))$ as $\xi \to 0$.
\end{lemma}

\begin{proof} We have
$$
\Re(\widehat{v_0}(\xi)) = \int _{\R^N} \cos (x\cdot \xi) v_0 (x)\, dx = 
-\int _{\R ^{N}}\sin(x'\cdot \xi ')\sin(x_N\xi_N)v_0(x)\,dx
$$
thanks to the fact that $v_{0}(x',-x_{N}) = -v_{0}(x',x_{N})$ according to \eqref{eq:v0-odd}. Since 
$$\int_{\R^N}\vert x_N v_{0}(x)\vert dx=2m_1(v_0) <+\infty, \quad\text{and}\quad
|\sin(x'\cdot \xi ')\sin(x_{N}\xi_{N})| \leq |x_{N}\xi_{N}| ,$$
using the Lebesgue dominated convergence theorem we deduce that, as $\xi \to 0$,
$\Re (\widehat{v_0}(\xi)) = o(\xi_{N})$. Similarly, for the imaginary part of $\widehat{v_{0}}(\xi)$ we write
\begin{eqnarray*}
-\Im(\widehat{v_0}(\xi)) & = & \int _{\R^N} \sin (x\cdot \xi) v_0 (x)\, dx = \int _{\R ^{N}}\cos(x'\cdot \xi ')\sin(x_N\xi_N)v_0(x)\,dx\\
& = & \int_{\R ^{N}} \sin(x_N\xi_N)v_0(x)\,dx +\int_{\R ^{N}} (\cos(x'\cdot \xi')-1))\sin(x_N\xi_N)v_0(x)\,dx\\
& = & \int_{\R ^{N}} \sin(x_N\xi_N)v_0(x)\,dx + o(\xi_{N}),
\end{eqnarray*}
thanks to the fact that $|(\cos(x'\cdot \xi') - 1)\sin(x_{N}\xi_{N})| \leq 2|x_{N}\xi_{N}|$, and then applying the Lebesgue dominated convergence theorem. We pursue with
\begin{eqnarray*}
-\Im(\widehat{v_0}(\xi)) & = & \int _{\R ^{N}} x_N\xi_N v_0(x)\,dx +\int _{\R ^{N}} (\sin(x_N\xi_N)-x_N\xi_N) v_0(x)\,dx + o(\xi_{N})\\
& = & \int _{\R ^{N}} x_N\xi_N v_0(x)\,dx + o(\xi_{N}) = \xi_{N}(2m_{1}(v_{0}) + o(1)),
\end{eqnarray*}
where we have used again the dominated convergence theorem. 
\end{proof}

Now, we estimate the rate of decay of the $L^{\infty}$ norm of the solution to the linear Cauchy problem \eqref{cauchy-diffusion} under consideration.

\begin{lemma}[Uniform estimate from above]\label{lem:decrease} There exists a positive constant $C$ such that, for any initial data $v_0$ satisfying \eqref{eq:v0-odd}, the solution of the Cauchy problem \eqref{cauchy-diffusion} satisfies
$$
\forall t >0, \qquad \|v(t,\cdot)\|_{L^\infty}\leq C\, \frac{m_1(v_0)+\|\widehat{v_0} \|_{L^{1}}}{(1+t)^{(N+1)/\beta}}.
$$
\end{lemma}

\begin{proof} From the convolution formula \eqref{convolution}, the Plancherel formula and \eqref{fourier-kernel}, we have
\begin{eqnarray*}
(2\pi)^N v(t,x)&=&\int _{\R^{N}} {\rm e}^{t(J_0(\xi)-1)}
\overline{\FF(v_0(x-\cdot))}(\xi)\,d\xi = \int _{\R^{N}} {\rm e}^{t(J_0(\xi)-1)} {\rm e}^{ix\cdot \xi}\overline{\FF(v_0)(-\xi)}\,d\xi\\
&=&\int _{\R^{N}} {\rm e}^{t(J_0(\xi)-1)} {\rm e}^{ix\cdot \xi}\widehat{v_0}(\xi)\,d\xi.
\end{eqnarray*}
Thanks to \eqref{eq:J0}, we can fix $r > 0$ small enough such that
$$
| \xi| \leq r \quad \Longrightarrow \quad J_0(\xi)-1\leq -\frac{1}{2} a| \xi | ^{\beta}.
$$
On the other hand, using \eqref{eq:J0-infini}, we know that there exists $\delta > 0$ such that
$$
| \xi| \geq r \quad \Longrightarrow \quad J_0(\xi)-1\leq -\delta.
$$
As a result
\begin{equation}\label{couper-1}
(2\pi)^N | v(t,x)| \leq \int_{| \xi | \leq r} 
{\rm e}^{-\frac 12 at| \xi| ^{\beta}}| \widehat{v_0}(\xi)| \, d\xi 
+ {\rm e}^{-\delta t}\|\widehat {v_0}\|_{L^{1}}.
\end{equation}
Setting
$$f_1(t) := \int_{| \xi | \leq r} 
{\rm e}^{-\frac 12 at| \xi| ^{\beta}}| \widehat{v_0}(\xi)| \, d\xi,$$
we use the change of variable $z = \xi t^{\frac 1\beta}$ in the integral and thus
$$
t^{\frac N \beta}f_1(t)=\int _{| z| \leq rt^{ 1/\beta }} e^{-\frac 12 a| z| ^{\beta}} \left | \widehat{v_0} \left(\frac{z}{t^{1/\beta}}\right)\right |\,dz.
$$
Upon letting $t\to +\infty$ in the above equality, using Lemma \ref{claim} and the dominated convergence theorem, we infer that, as $t\to+\infty$,
\begin{equation}\label{f11}
t^{\frac N \beta}f_1(t)\sim \int _{\R^N} e^{-\frac 12 a| z| ^{\beta}}  \frac{| z_N| }{t^{1/\beta}}2 m_1(v_{0})\, dz=:\frac{C m_1(v_{0})}{t^{1/\beta}},
\end{equation}
for some constant $C > 0$ depending only on $a$, $N$. In view of \eqref{couper-1} and \eqref{f11}, the lemma is proved.
\end{proof}

An estimate from below on $v$ is more tricky to obtain. Indeed since $v(t,0) = 0$ for all times, we need to evaluate $v(t,\cdot)$ at an appropriate moving point, so that we may obtain a lower bound on the $L^\infty$ norm with the correct order, namely $t^{-(N+1)/\beta}$.

\begin{lemma}[Pointwise estimate from below]\label{lem:from-below} There exist two constants $\gamma > 0$ and $C= C(\gamma) > 0$ such that the following holds: for any $v_0 \in L^1(\R^N)$ satisfying \eqref{eq:v0-odd} and such that $\widehat{v_0}\in L^{1}(\R^{N})$, there is $t_0 > 0$ such that the solution of the Cauchy problem \eqref{cauchy-diffusion} satisfies (here ${\bf e}_{N} \in \R^N$ denotes the unit vector in the $x_{N}$ direction)
$$
v\left(t,\gamma t^{1/\beta}{\bf e}_N\right)\geq \frac{Cm_1(v_0)}{t^{(N+1)/\beta}},\quad \forall t\geq t_0.
$$
\end{lemma}

\begin{proof} Thanks to the Lebesgue dominated convergence theorem, for $\gamma > 0$ we have
\begin{equation}\label{eq:C1-gamma}
C_1(\gamma) := \int _{\R^N}{\rm e}^{-a| z| ^{\beta}}z_N\sin(\gamma z_N)dz\sim \gamma \int _{\R ^N} {\rm e}^{-a| z| ^{\beta}}z_N^2 dz \quad \text{ as } \gamma \to 0,
\end{equation}
so that there exists $\gamma_{*} > 0$ small enough such that for $0 < \gamma \leq \gamma_{*}$ we have  $C_1(\gamma) > 0$. In the remainder of this proof we assume that $0 < \gamma \leq \gamma_{*}$.

For a function $b := b(t)>0$ to be chosen appropriately later on, we compute
$$
v(t,b(t){\bf e}_N)= \int _{\R ^N}\!\!\! G(t,y)\,v_0\left(b(t){\bf e}_N - y\right)dy = 
(2\pi)^{-N}\int _{\R^N} \!\!\! {\rm e}^{t(J_0(\xi)-1)} {\rm e}^{ib(t)\xi _N}\widehat{v_0} (\xi)\, d\xi,
$$
where, as above, we have used the Plancherel formula \eqref{eq:Plancherel} and $\overline{\widehat{v_0}(-\xi)}=\widehat{v_0}(\xi)$. 
Since $J_{0}(\xi)$ and $v(t,b(t)e_N)$ are real valued, this reduces to
$$
(2\pi)^N v(t,b(t){\bf e}_N) \!= \!\!\int _{\R^N}\!\!\! {\rm e}^{t(J_0(\xi)-1)} 
\!\Big[\cos(b(t)\xi_N) \Re(\widehat{v_0}(\xi)) 
- \sin (b(t)\xi _N) \Im(\widehat{v_0}(\xi))\Big]d\xi.
$$
Thanks to the property \eqref{eq:J0} and Lemma \ref{claim}, we can fix $r > 0$ small enough such that
\begin{equation}\label{pres}
| \xi| \leq r \quad \Longrightarrow \quad J_0(\xi)-1\leq - \frac 12 a| \xi | ^{\beta},
\end{equation}
and
\begin{equation}\label{pres2}
| \xi| \leq r \quad\Longrightarrow\quad |\widehat{v_0}(\xi)| \leq 3 m_1| \xi_N|.
\end{equation}
On the other hand, thanks to \eqref{eq:J0-infini}, there is $\delta >0$ such that
\begin{equation}
\label{loin}
| \xi| \geq r \quad\Longrightarrow \quad J_0(\xi)-1\leq -\delta.
\end{equation}
Now, we write $(2\pi)^N v(t,b(t){\bf e}_N) = f_{1}(t) + f_{2}(t)$ where we wet first
\begin{equation}\label{couper}
f_{1}(t) := \int _{|\xi| \leq r} \!\!\! {\rm e}^{t(J_0(\xi)-1)} 
\Big[\cos(b(t)\xi_N) \Re(\widehat{v_0} (\xi)) 
- \sin (b(t)\xi _N) \Im(\widehat{v_0}(\xi))\Big]\,d\xi
\end{equation}
so that 
$$f_{2}(t) := \int _{|\xi| > r} {\rm e}^{t(J_0(\xi)-1)} 
\Big[\cos(b(t)\xi_N) \Re(\widehat{v_0}(\xi)) 
- \sin (b(t)\xi _N) \Im(\widehat{v_0}(\xi))\Big]\,d\xi.$$
Using \eqref{loin} we infer that
\begin{equation}\label{f2}
| f_2(t)|  \leq  2\int _{| \xi| > r}  \E | \widehat{v_0}(\xi)|\, d\xi
\leq  2e^{-\delta t} \|\widehat {v_0} \Vert_{L^{1}}.
\end{equation}
Next, we use the change of variable $z =t^{1/\beta} \xi$ in $f_1(t)$ and we obtain
\begin{eqnarray*}
t^{N/\beta}f_1(t)&=&\int _{| z| \leq rt^{ 1/\beta }} 
{\rm e}^{t\left(J_0(t^{-1/\beta}z)-1\right)} 
\Big[\cos(t^{-1/\beta}b(t) z_{N}) \,
\Re(\widehat{v_0}(t^{-1/\beta}z)) \\
&&\qquad\qquad\qquad\qquad\qquad\qquad
-\sin(t^{-1/\beta}b(t)z_{N}) \, 
\Im(\widehat{v_0}(t^{-1/\beta}z))\Big]\,dz.
\end{eqnarray*}
At this point one sees that making the appropriate choice $b(t) := \gamma t^{1/\beta}$ we have
\begin{eqnarray}
\hskip 1cm t^{N/\beta}f_1(t) &=&
\int _{| z| \leq rt^{ 1/\beta }}
{\rm e}^{t\left(J_0(t^{-1/\beta} z) - 1\right)} 
\Big[\cos\left(\gamma z_N\right) \, 
\Re(\widehat{v_0}(t^{-1/\beta}z))
\label{morceau-proche-zero}\\
&&\qquad\qquad\qquad\qquad\qquad\qquad
-\sin \left(\gamma z_N\right) \, 
\Im(\widehat{v_0}(t^{-1/\beta}z))\Big]\,dz.\nonumber
\end{eqnarray}
We now let $t\to +\infty$ in \eqref{morceau-proche-zero}: since we have \eqref{pres} and \eqref{pres2}, we can use the Lebesgue dominated convergence theorem so that we obtain, thanks to \eqref{eq:C1-gamma} and Lemma \ref{claim},
\begin{equation}\label{f1}
t^{N/\beta}f_1(t)\sim 
\int _{\R^N} {\rm e}^{-a| z| ^{\beta}}
\sin(\gamma z_N) t^{-1/\beta}z_{N}\, m_1(v_{0})\, dz = 
C_1(\gamma) t^{-1/\beta} m_{1}(v_{0}),
\end{equation}
as $t\to+\infty$. In view of \eqref{couper}, \eqref{f2} and \eqref{f1}, the lemma is proved.
\end{proof}

\subsection{$L^1$ estimates for the linear diffusion equation} For the proof of systematic blow-up in the critical case $\alpha=\beta/(N+1)$, namely Theorem \ref{th:systematic-1}, we will need further estimates, in particular of the $L^1$-type. It is customary to obtain $L^1$-type estimates for the solutions of a diffusion equation such as
$$\partial_{t} v = \Delta v \quad\text{in }\, (0,+\infty)\times \R^N,
\qquad v(0,x) = v_{0}(x),$$
by considering first the case $v_{0} \geq 0$ and by multiplying the equation by a truncation function $\zeta_{R}(x) := \rho(x/R)$, where $\rho \in C^\infty_{c}(\R^N)$ is such that $\mathbf 1_{\{|x| \leq 1\}} \leq \rho \leq \mathbf 1_{\{|x| \leq 2\}}$, to obtain, after integrating by parts and then over $[0,T]$,
$$\int_{\R^N}v(T,x)\zeta_{R}(x)dx = 
\int_{\R^N}v_{0}(x)\zeta_{R}(x)dx 
+ \int_{0}^T\!\!\!\!\int_{\R^N}v(t,x) \Delta \zeta_{R}(x)\,dxdt.$$
Then, since $|\Delta\zeta_{R}(x)| \leq cR^{-2}\mathbf 1_{\{R\leq |x| \leq 2R\}}$ for some constant $c>0$, one deduces that the $L^1$ norm of $v(t,\cdot)$ is bounded (in fact constant, in this case). The same can be done with the norm of $x_{N}v$ in $L^1$. This is made possible because the Laplacian $\Delta$ is a local operator and obviously we have
\begin{equation}\label{analogous1}
\Delta(x_{N}\zeta_{R}) \equiv 0 \qquad\text{for }\, |x| \geq  2R,
\end{equation}
and 
\begin{equation}\label{analogous2}
\Delta(x_{N}\zeta_{R}) \leq c |x_{N}| R^{-2} \qquad \text{for }\, |x| \leq 2R.
\end{equation}
However, for a non-local operator,  in particular when dealing with an equation in a half-space, we will have to use the following estimates involving the operator $A$.

\begin{lemma}[Estimates to substitute to \eqref{analogous1} and \eqref{analogous2} in the general case]\label{lem:A-zeta-R}
Assume that $A$ satisfies the conditions of Theorem \ref{th:systematic-1}. 
Let $\rho \in C_c^\infty(\R)$ be such that 
$$\mathbf 1_{[-1,+1]} \leq \rho \leq \mathbf 1_{[-2,+2]},\qquad \rho(-s) = \rho(s),$$
and define the truncation function $\zeta_{R}$ for $R >0$ by setting
$$ \zeta _R(x) := \rho\left(\frac{|x|}{R}\right).$$
Then for a constant $C > 0$ independent of $R$ we have
\begin{eqnarray}
& x_{N}\left(A\left(x_{N}\zeta_{R}\right)\right) \geq 0 & \qquad \text{for }\, |x| \geq 2R \label{eq:xN-A-zeta-1} \\
& \left|\left(A\left(x_{N}\zeta_{R}\right)\right)\right| \leq C\, \displaystyle\frac{|x_{N}|}{R^\beta}
&\qquad \text{for }\, x \in \R^N .\label{eq:xN-A-zeta-2}
\end{eqnarray}
\end{lemma}

\begin{proof}
Since the case of the Laplacian (for which $\beta=2$) is clear from \eqref{analogous1} and \eqref{analogous2}, we may only consider the case of the operator $A$ being given by $Av = J*v - v$ which, up to working in the principal value sense, also covers the case of the fractional Laplacian $-(-\Delta)^{\beta/2}$ by selecting $J(z)=\frac{C}{\vert z\vert^{N+\beta}}$ where $C=C(\beta,N)>0$ is a known constant. 

We begin by noting that thanks to \eqref{eq:Cond-J-cri}, and the fact that $z_{N} \mapsto J(z',z_{N})$ is even for any fixed $z' \in \R^{N-1}$, for any $x,y \in \R^N$ with $x_{N} \geq 0$, $y_{N} \geq 0$ we have
\begin{equation}\label{eq:Cond-J-cri-2}
J(x'-y',x_N-y_N) - J(x'-y',x_N+y_N)+J(x'+y',x_N-y_N)-J(x'+y',x_N+y_N) \geq 0.
\end{equation}
For clarity of the exposition, let $\phi(x) := x_{N}$. Then for $|x| \geq 2R$ we have 
$$A(\phi\zeta_{R}) = J*(\phi\zeta_{R}) - \phi\zeta_{R} = J*(\phi\zeta_{R})$$ 
and since $\phi(-y)\zeta_{R}(-y) = -\phi(y)\zeta_{R}(y)$ we have
$$A(\phi\zeta_{R})(x) = \int_{\R^N}J(x-y)\phi(y)\zeta_{R}(y)dy = -\int_{\R^N}J(x+y)\phi(y)\zeta_{R}(y)dy.$$
Hence, for $|x| \geq 2R$,  we may write
\begin{eqnarray}
A(\phi\zeta_{R}) (x)& = & 
\frac{1}{2} \int_{\R^N}\left[J(x-y) - J(x+y)\right]\phi(y)\zeta_{R}(y)dy \nonumber\\
&=&\frac{1}{2} \int_{\vert y\vert <2R}\left[J(x'-y',x_N-y_N) - J(x'+y',x_N+y_N)\right]\phi(y)\zeta_{R}(y)dy \nonumber\\
 & = &
\frac{1}{2}\int_{[|y|<2R]\cap[y_{N} > 0]} [J(x'-y',x_N-y_N) -J(x'-y',x_N+y_N)\label{eq:A-phi-zeta} \\
&&\qquad\qquad +J(x'+y',x_N-y_N)-J(x'+y',x_N+y_N)]\phi(y)\zeta_{R}(y)dy. \nonumber
\end{eqnarray}
Now, using \eqref{eq:Cond-J-cri-2}, we see that, when $x_{N} \geq 0$, the function under the integral sign on the right hand side of \eqref{eq:A-phi-zeta} is nonnegative, so that $A(\phi\zeta_{R})(x) \geq 0$ when $x_{N} \geq 0$ and $|x|\geq 2R$. However, \eqref{eq:A-phi-zeta} shows also that $A(\phi\zeta_{R})$ is odd with respect to $x_N$ on $[|x| \geq  2R]$, so that we may infer that $\phi(x) A(\phi\zeta_{R})(x) \geq 0$ when $|x| \geq 2R$, and the proof of \eqref{eq:xN-A-zeta-1} is complete.
\medskip

In order to show \eqref{eq:xN-A-zeta-2} let us denote 
$$\zeta(x) := \rho(|x|),\quad\text{so that }\, \zeta_{R}(x) = \zeta(x/R), \qquad \psi_{R}(x) := x_{N}\zeta_{R}(x).$$
We know that
$$\mathscr{F}(\zeta_{R})(\xi) = R^N\mathscr{F}(\zeta)(R\xi), \qquad
\mathscr{F}(\psi_{R})(\xi) = {\rm i}\frac{\partial}{\partial \xi_{N}}\mathscr{F}(\zeta_{R})(\xi),$$
and as a consequence we have 
\begin{equation}\label{eq:Fourier-psi}
\widehat{\psi_{R}}(\xi) = {\rm i}R^{N+1}\left(\frac{\partial}{\partial \xi_{N}}\widehat{\zeta}\right)(R\xi),
\end{equation}
and we note that $\xi_{N} \mapsto \widehat{\psi_{R}}(\xi',\xi_{N})$ is odd. 

Now, using the fact that $\mathscr{F}(A(\psi_{R})) =\mathscr{F}(J*\psi_{R} - \psi_{R})= (\widehat{J}(\xi) - 1)\widehat{\psi_{R}}(\xi)$, as well as the fact that $\widehat{\psi_{R}}$ is odd with respect to $\xi _N$ and the property \eqref{eq:Fourier-psi} we have\begin{eqnarray*}
(2\pi)^N A(\psi_{R}) (x) & = &
\int_{\R^N}\!\!\!\! {\rm e}^{{\rm i}x\cdot\xi} (\widehat{J}(\xi) - 1)\widehat{\psi_{R}}(\xi) d\xi 
\\
&=& {\rm i}\int_{\R^N}\!\!\!
{\rm e}^{{\rm i}x'\cdot\xi'}\sin(x_{N}\cdot\xi_{N}) (\widehat{J}(\xi) - 1)\widehat{\psi_{R}}(\xi) d\xi
\\
& = &
- R^{N+1} \int_{\R^N}\!\!\!\!
{\rm e}^{{\rm i}x'\cdot\xi'}\sin(x_{N}\cdot\xi_{N}) (\widehat{J}(\xi) - 1)\left(\frac{\partial}{\partial\xi_{N}} \widehat{\zeta}\right)(R\xi) d\xi \\
& = &
- R \int_{\R^N}\!\!\!\!
{\rm e}^{{\rm i}R^{-1} x'\cdot\xi'}\sin(R^{-1}x_{N}\cdot\xi_{N})  (\widehat{J}(R^{-1}\xi) - 1)\left(\frac{\partial}{\partial\xi_{N}} \widehat{\zeta}\right)(\xi) d\xi.
\end{eqnarray*}
Since on the one hand $R |\sin(R^{-1}x_{N}\cdot \xi_{N})| \leq |x_{N}\xi_{N}|$, and on the other hand, for a fixed $\xi \in \R^N$, thanks to \eqref{eq:J0} we have 
$$\widehat{J}(R^{-1}\xi) - 1 = - aR^{-\beta}\vert \xi\vert + o(R^{-\beta})$$ 
as $R \to +\infty$, using the Lebesgue dominated convergence theorem and the fact that $\frac{\partial}{\partial \xi _{N}}\widehat{\zeta} \in \mathscr{S}(\R^N)$, we deduce that 
$$(2\pi)^N |A(\psi_{R})(x)| \leq C |x_{N}| R^{-\beta},$$
as $R \to +\infty$. Thus the proof of \eqref{eq:xN-A-zeta-2} is complete.
\end{proof}

We conclude these preliminaries with the following.

\begin{lemma}[Taking advantage of the self-adjointess]\label{lem:xN-Au}
Assume that $A$ satisfies the conditions of Theorem \ref{th:systematic-1}. 
Let $u \in L^1(\R^N)$ be such that 
$$\int_{\R^N}|x_{N}u(x)| dx < +\infty, \qquad
\int_{\R^N}|x_{N}Au(x)| < +\infty.$$
Then we have
$$\int_{\R^N}x_{N}(Au)(x) dx = 0.$$
\end{lemma}

\begin{proof}
Indeed, with the notations used in the proof of the above lemma, since $A$ is self-adjoint, we have
$$\int_{\R^N}x_{N}\zeta_{R}(x)(Au)(x)dx = \int_{\R^N}u(x)(A\psi_{R})(x)dx.$$
However using \eqref{eq:xN-A-zeta-2} we have $|(A\psi_{R})(x)| \leq C |x_{N}|R^{-\beta}$, and thus the right hand side of the above equality converges to zero as $R \to +\infty$. On the other hand it is clear that we have
$$0 =  \lim_{R\to +\infty}\int_{\R^N}x_{N}\zeta_{R}(x)(Au)(x)dx = \int_{\R^N}x_{N}(Au)(x)dx$$
since $\int_{\R^N}|x_{N}(Au)(x)|dx < +\infty$.
\end{proof}

\section{Cauchy problems in the half-space}\label{s:cauchy}

\subsection{The linear diffusion equation in $\R^N_+$} The resolution of the linear diffusion equation in the half-space $\R_{+}^N$, in particular in the case of a non local operator such as $A u := J*u - u$, is made possible thanks to   Assumption \ref{ass:diffusion}.

More precisely, for a function $w_0 \in L^1(\R_{+}^N)$ we define ${\widetilde w}_{0} \in L^1(\R^N)$ as being
\begin{equation}\label{eq:Odd-extension}
{\widetilde w}_{0}(x) := \mathbf 1_{\{x_{N} > 0\}}(x)w_{0}(x',x_{N}) - \mathbf 1_{\{x_{N} < 0\}}(x) w_{0}(x',-x_{N}).
\end{equation}
Then, thanks to Assumption \ref{ass:diffusion}, the function
$${\widetilde w}(t,\cdot) := G(t)*{\widetilde w}_{0}$$
is well-defined, solves the equation
\begin{equation}\label{eq:Cauchy-Space}
\begin{cases}
\partial_{t} {\widetilde w} = A{\widetilde w} &\text{ in } (0,+\infty)\times \R^{N}\\
{\widetilde w}(0,\cdot) = {\widetilde w}_0 &\text{ in } \R ^{N},\\
\end{cases}
\end{equation}
and satisfies
$${\widetilde w}(t,x',-x_{N}) = - {\widetilde w}(t,x',x_{N}).$$
Now we may define the solution of the evolution equation
\begin{equation}\label{eq:Cauchy-Half-Space}
\begin{cases}
\partial_{t} w = Aw &\text{ in } (0,+\infty)\times \R ^{N}_+\\
w(0,\cdot) = w_{0} &\text{ in } \R ^{N}_+\\
w(t,x',0) = 0 &\text{ on } \partial \R ^{N}_+,
\end{cases}
\end{equation}
as being
\begin{equation}\label{eq:Def-Sol-Half-Space}
\forall x_{N} >0, \qquad w(t,x',x_{N}) := {\widetilde w}(t,x',x_{N}).
\end{equation}

Thus, when the kernel $G(t)$ is given by a function belonging to $L^1(\R^N)$, including the case of the fractional Laplacian up to working in the principal value sense and the case of the convolution operators as in Example \ref{ex:nonreg} up to a harmless Dirac mass (see \eqref{G-measure}), the solution $w$ of \eqref{eq:Cauchy-Half-Space} can be written as 
\begin{equation}
\label{eq:Def-Sol-Half-Space-Conv}
\forall (t,x) \in (0,+\infty)\times\R_{+}^N, \quad w(t,x) := \int _{\R^{N}_+} K(t,x,y)w_0(y)dy,
\end{equation}
where the kernel $K$ is defined by
\begin{equation}\label{eq:Def-K}
K(t,x,y) := G(t,x'-y',x_N-y_N) - G(t,x'-y',x_N+y_N),
\end{equation}
for any $x=(x',x_N)\in \R ^{N}_+$, $y=(y',y_N)\in \R ^{N}_+$ and $t > 0$. We point out that, in this case, in order to ensure the validity of the comparison principle for \eqref{eq:Cauchy-Half-Space}, one must assume also that, for all $t >0$ and $x' \in \R^{N-1}$,
\begin{equation}\label{cond-G-positivite}
\text{ the function $x_{N} \mapsto G(t,x',x_{N})$ is nonincreasing on $(0,+\infty)$},
\end{equation}
so that for $t >0$ and $x,y \in \R_{+}^N$ one has $K(t,x,y) \geq 0$.

\begin{rem}[On the comparison principle assumption]\label{rem:nouvelle}
It is rather clear that condition \eqref{cond-G-positivite} is fulfilled in the case of the Laplacian and the fractional Laplacian.

On the other hand, let us consider a convolution operator as stated in Example \ref{ex:nonreg}, with the additional assumption \eqref{eq:Cond-J-cri}. Since convolution preserves the class of radially nonincreasing kernels ---see Lemma \ref{lem:convolution}---  it follows from \eqref{eq:Cond-J-cri} that $G$, given by \eqref{G-measure}, satisfies \eqref{cond-G-positivite}.\qed
\end{rem}

The following result seems to be more or less ``well-known'' but, since we did not find it in the literature, for the sake of completness we give its proof.

\begin{lemma}[Convolution of radially nonincreasing kernels]\label{lem:convolution} Convolution preserves the class of radially nonincreasing kernels.
\end{lemma}

\begin{proof} Let $K,J \in L^1(\R^N)$ be nonnegative and  radially nonincreasing. It is easy to see that $H := K*J$ is radially symmetric. Since for instance $H(x) = H((|x|,0,...,0))$, in order to show that $H$ is radially nonincreasing it is enough to show the result when $N = 1$.

Now let  $K,J \in L^1({\R})$ be two even functions which are nonincreasing on $(0,+\infty)$. We first consider the case where $K\in C^1({\R})$ and that $K' \in L^1({\R})$. Then $H' = K' * J$ and, for $x > 0$, 
\begin{eqnarray*}
H'(x) & = & \int_{-\infty}^{+\infty} K'(x-y)J(y)dy\\
& = & \int_{-\infty}^{0} K'(y)J(x-y)dy + \int_0^{+\infty} K'(y)J(x-y)dy\\
&=& \int_0^{+\infty} K'(-y)J(x+y)dy + \int_0^{+\infty} K'(y) J(x-y)dy\\
&=& \int _0^{+\infty} K'(y) (J(|x-y|)-J(x+y))dy,
\end{eqnarray*}
where we have used the fact that $J(x-y) = J(|x-y|)$.
Clearly for $x,y > 0$ we have $J(|x-y|) - J(x+y) \geq 0$, since $J$ is nonincreasing on $(0,+\infty)$ and $|x-y| \leq x + y$, while $K'(y) \leq 0$. Therefore $H'(x) \leq 0$ for $x > 0$.

In the general case, we use a regularization argument. Consider the heat kernel $G_{t}(x) := (4\pi t)^{-1/2}\exp(-x^2/4t)$ for $t > 0$. Since $G_{t} \in \mathscr{S}(\R)$, is even and decreasing on $(0,+\infty)$, the function $K_{t} := G_{t}*K$ is even and nonincreasing on $(0,+\infty)$ according to the above result. As a consequence, $H_{t} := K_{t}*J = G_{t}*K*J$ is also even and nonincreasing on $(0,+\infty)$. Since $G_{t}*(K*J)$ converges to $K*J$ as $t \to 0^+$ in $L^1(\R)$,  we know that, for a subsequence $(t_{n})_{n \geq 1}$,  $G_{t_{n}}*(K*J) \to K*J$ a.e.\ on $\R$. This shows that $K*J$ is also even and noncreasing on $(0,+\infty)$.
\end{proof}

\subsection{The semilinear equations in $\R^{N}$ and $\R^{N}_+$}

We first need to say a word on the notion of solutions to
\begin{equation}\label{eq:Cauchy-Space-NL}
\begin{cases}
\partial_{t} u = Au + |u|^{\alpha}u &\text{ in } (0,T)\times \R^{N}\\
u(0,\cdot) = u_0 &\text{ in } \R^{N}.
\end{cases}
\end{equation}
We shall say that, for some $T > 0$, a function 
$$u\in C^1((0,T),L^1(\R ^N)) \cap C\left([0,T],L^1(\R^N)\cap L^{\alpha+1}(\R^N)\right) $$
 is  a weak solution to \eqref{eq:Cauchy-Space-NL} if for every $\phi \in C^1([0,T],\mathscr{S}(\R^N))$ we have 
$$\int_{0}^T\!\!\!\int_{\R^N} \!\!\!\left[
u(t,x)\left(\partial_{t}\phi + A\phi\right) 
+ |u|^{\alpha} u\phi\right]dxdt 
= \int_{\R^N}\!\!\! \left[u(T,x)\phi(T,x) - u_{0}(x)\phi(0,x)\right] dx.$$
For such solutions, the comparison principle is available. Also, for $u_0\in L^1(\R ^N) \cap L^\infty(\R^N)$, the associated Cauchy problem \eqref{eq:Cauchy-Space-NL} admits a unique solution defined on some maximal interval $[0,T)$. Moreover either $T = +\infty$ and the solution is global, or $T < +\infty$ and then $\|u(t,\cdot)\|_{L^\infty}$ tends to $+\infty$ as $t\to T$, which is called blow-up in finite time. These facts are rather well-known, and parts of them can be found in  \cite{Gar-Qui-10} for instance. 

As for the semilinear Cauchy problem  in the half-space \eqref{eq:1}, the strategy consists, as above, in constructing the odd (with respect to $x_N$) extension of the initial data $u_0$ so that we are back to \eqref{eq:Cauchy-Space-NL}. Details are omitted.

\section{Systematic blow-up when $0 < \alpha < \beta/(N + 1)$} \label{s:blowup}

In this section, we prove the blow-up of any solution when $0 < \alpha < p_F - 1$, as stated in Theorem \ref{th:systematic}. 

\medskip

If the solution does not blow up in finite time we would have a global solution $u(t,x) \geq 0$ of the Cauchy problem \eqref{eq:1}, with an initial data $u_{0}$ satisfying Assumption \ref{ass:initial}, and moreover since we assume that $u_{0} \in L^\infty(\R^N)$, for any $T > 0$ we have also $u \in L^\infty(0,T; L^\infty(\R^N))$. 

Upon considering ${\widetilde u}_{0}$, the odd extension of $u_{0}$ through the definition \eqref{eq:Odd-extension}, we can define, for any $t\geq 0$, the  function
\begin{equation}
\label{def:f}
f(t) := 
\int_{\R ^{N}}G(t,x){\widetilde u}_0\left(x+\gamma t^{1/\beta}{\bf e}_N\right)dx 
= v(t,\gamma t^{1/\beta}{\bf e}_{N}),
\end{equation}
where $v$ is the solution of the linear equation
\begin{equation}\label{eq:Cauchy-Half-Space-2}
\begin{cases}
\partial_{t} v = Av &\text{ in } (0,+\infty)\times \R ^{N}_+\\
v(0,\cdot) = u_{0} &\text{ in } \R ^{N}_+\\
v(t,x',0) = 0 &\text{ on } \partial \R ^{N}_+,
\end{cases}
\end{equation}
and where we recall, see \eqref{parite-kernel}, that the kernel $G(t)$ is even in the variable $x$. Notice that in \cite{Alf-fujita} by the first author it was enough to define the above quantity with $\gamma = 0$, but here this would yield $f(t)=0$: that is why we need to follow the value of $v$ at an appropriate moving point, namely at $\gamma t^{1/\beta}{\bf e}_N$. 

From \eqref{convolution} and \eqref{parite-kernel}, we remark that, thanks to the  comparison principle, we have $u\geq v$ in $\R ^{N}_+$. In particular, it follows from  Lemma \ref{lem:from-below} that the parameter $\gamma$ being chosen as in Lemma \ref{lem:from-below}, for some $t_{0} > 0$ and a constant $C_{*}(\gamma)$ we have
\begin{equation}
\label{vers-absurdite}
\forall t \geq t_{0}, \qquad f(t) = v\left(t,\gamma t^{1/\beta}{\bf e}_N\right) \geq \frac{C_{*}(\gamma)m_1(u_0)}{t^{(N+1)/\beta}}.
\end{equation}
Now, assuming that the solution $u$ of the nonlinear equation is global, using the nonlinear term, we are going to estimate the quantity $f(t)$ from above.

\begin{lemma}[Estimate from above] \label{lem:f-above}
Assume that the solution $u \not\equiv 0$ of \eqref{eq:1} is nonnegative and exists for all times $T > 0$. Then, for any $\gamma > 0$, the function $f$ being defined by \eqref{def:f}, there exists a constant $C^*(\alpha) > 0$,  such that for all $t > 0$ we have
\begin{equation}
\label{f-above}
f(t)\leq \frac{C^*(\alpha)}{(1+t)^{1/\alpha}}, \quad \text{ for any } t > 0.
\end{equation}
\end{lemma}

\begin{proof} Let $T>0$ be given. Fix some $0 < t\leq T$, and for $0 < s < t$ define the function
\begin{eqnarray}
g(s) & := & \int _{\R^{N}}G(t-s,x){\widetilde u}(s,x+\gamma t^{1/\beta}{\bf e}_N)dx\nonumber\\
& = & \int _{\R^{N}_+}K(t-s,y,\gamma t^{1/\beta}{\bf e}_{N})u(s,y)dy,\nonumber
\end{eqnarray}
where the kernel $K$ is defined in \eqref{eq:Def-K}. Observe that if $v$ is defined in \eqref{eq:Cauchy-Half-Space-2}, we have 
$$g(0) = v\left(t,\gamma t^{1/\beta}{\bf e}_N\right) = f(t).$$
In view of equation \eqref{eq:1}, we know that ${\widetilde u}$, the odd extension of $u$ defined through \eqref{eq:Odd-extension}, satisfies the equation
$$\partial_{t}{\widetilde u} = A{\widetilde u} + |{\widetilde u}|^\alpha {\widetilde u}\quad \text{ in }\, (0,+\infty)\times\R^N,$$
and thus for $0 < s < t$ we have
\begin{eqnarray}
g'(s) &=& \int _{\R^{N}} \left[
-AG(t-s,x){\widetilde u}(s,x+\gamma t^{1/\beta}{\bf e}_N) 
+ G(t-s,x)A{\widetilde u}(s,x+\gamma t^{1/\beta}{\bf e}_N)\right.
\nonumber\\
&&\quad\quad +\, \left. G(t-s,x)| {\widetilde u}| ^{\alpha}{\widetilde u}(s,x+\gamma t^{1/\beta}{\bf e}_N)\right]dx.\nonumber
\end{eqnarray}
Since the operator $A$ is self-adjoint, this reduces to
\begin{eqnarray}
g'(s) &=& 
\int _{\R^{N}}G(t-s,x)| {\widetilde u}|^{\alpha}{\widetilde u}(s,x+\gamma t^{1/\beta}{\bf e}_N)dx\nonumber\\
&=&
\int _{\R^{N}_+} K(t-s,y,\gamma t^{1/\beta}{\bf e}_N)u^{1+\alpha}(s,y)dy
\label{derivee-de-g}
\end{eqnarray}
using the fact that $u \geq 0$ on $\R^{N}_+$. We see that, for $t>0$,
$$
\int _{\R^{N}_+}K(t-s,y,\gamma t^{1/\beta}{\bf e}_N)dy = 1 - 2\int _{z_N\geq \gamma t^{1/\beta}}G(t,z)dz\in(0,1),
$$
so that the Jensen inequality yields
$$
g'(s) \geq  
\left(\int_{\R^{N}_+} K(t-s,y,\gamma t^{1/\beta}{\bf e}_N)dy\right)^{-\alpha}\, 
g^{1+\alpha}(s)\geq  g^{1+\alpha}(s).
$$
Integrating from $s=0$ to $s=t-\ep$ we get
$$
\frac{1}{f^\alpha(t)} = 
\frac{1}{g^{\alpha}(0)} \geq 
\frac{1}{g^\alpha(0)} - \frac{1}{g^{\alpha}(t-\ep)} \geq 
\alpha(t - \ep).
$$
Letting $\ep \to 0$ concludes the proof of Lemma \ref{lem:f-above}. 
\end{proof}

We are now in the position to complete the proof of Theorem \ref{th:systematic}.

\begin{proof}[Proof of Theorem \ref{th:systematic}]  

Recall that we are assuming that a global solution $u(t,x) \geq 0$ in $(0,+\infty)\times\R^N_{+}$ exists and that $u\not\equiv 0$, so that Lemma \ref{lem:f-above} holds for any $\gamma > 0$. Upon choosing $\gamma > 0$ as in Lemma \ref{lem:from-below}, from the estimates \eqref{vers-absurdite} and \eqref{f-above} we deduce that for $t \geq t_{0}$ large enough we have 
\begin{equation}\label{eq:bornes-1}
C_{*}(\gamma) m_{1}(u_{0}) t^{-(N+1)/\beta} \leq C^*(\alpha) (1+t)^{-1/\alpha},
\end{equation}
where $C_{*}(\gamma)$ and $C^*(\alpha)$ are two constants independent of $u_{0}$.
Letting $t \to +\infty$, the above inequality implies that $1/\alpha \leq (N+1)/\beta$. Therefore when 
$$0 < \alpha < p_F - 1 = \frac{\beta}{N+1}$$ 
we cannot have a nonnegative global solution $u$ of \eqref{eq:1}.
\end{proof}

The proof of the systematic blow-up of nonnegative solutions when $\alpha = \beta/(N+1)$ is more delicate, and the next section is devoted to its proof.

\section{Systematic blow-up when $\alpha = \beta/(N+1)$}\label{s:critical}

We now consider the critical case $\alpha = p_F - 1 = \beta/(N+1)$, for which, assuming that the solution $u$ is nonnegative and global, the inequality \eqref{eq:bornes-1} only provides that the initial data of any nonnegative global solution $u$ must satisfy
\begin{equation}\label{eq:cond-u0-cri}
m_1(u_0)\leq \frac{C^*(\alpha)}{C_{*}(\gamma)} =: C(\alpha,\gamma),
\end{equation}
where the constant $C(\alpha,\gamma) > 0$ depends only on the dimension $N$, on $\alpha,\gamma$ and on the operator $A$, but not on the initial data. Thus, by considering $u(\tau,\cdot)$ as an initial value for the evolution equation, we deduce that $u(t+\tau,\cdot)$ satisfies the inequalities \eqref{vers-absurdite} and \eqref{f-above}, so that $m_{1}(u(\tau,\cdot)) \leq C(\alpha,\gamma)$. Hence, the following holds, where we recall that $\widetilde u(t,\cdot)$ denotes the odd extension of $u(t,\cdot)$.

\begin{lemma}[A bound for the \lq\lq first moment'']\label{lem:M-1-bounded}
Assume that $\alpha = \beta/(N+1)$ and that the solution $u$ of \eqref{eq:1} is nonnegative and exists for all $T > 0$. Then denoting by $M_{1}(t)$
$$
M_1(t) := \int _{\R^{N}} x_N {\widetilde u}(t,x)dx,
$$
we have, for all $t\geq 0$,
\begin{equation}
\label{norme-L1}
0\leq M_1(t) \leq 2C(\alpha,\gamma).
\end{equation}
\end{lemma}

Now, we know that ${\widetilde u}$ satisfies
\begin{equation}\label{eq:1-RN}
\begin{cases}
\partial_t {\widetilde u} = A{\widetilde u} +  |{\widetilde u}|^{\alpha}{\widetilde u} \quad &\mbox{in }\, (0,T)\times \R^{N}, \\
{\widetilde u}(0,x) = {\widetilde u}_{0}(x)\quad & \mbox{for }\, x \in \R^{N}.
\end{cases}
\end{equation}
Assuming that $A$ satisfies the conditions of Theorem \ref{th:systematic-1}, multiplying \eqref{eq:1-RN} by $x_N$, integrating  over $x\in \R^N$, and using Lemma \ref{lem:xN-Au}, we get
$$
\frac{d}{dt}M_1(t) = \int_{\R ^N} x_N  |{\widetilde u}|^{\alpha}{\widetilde u}(t,x)\,dx,
$$
so that
$$
\int_0^{T}\int_{\R ^N} x_N |{\widetilde u}|^{\alpha}{\widetilde u}(t,x)\,dxdt = M_1(T) - M_1(0)\leq 2C(\alpha,\gamma),
$$
for all $T\geq 0$. As a result, we infer that (since $x_N {\widetilde u}(t,x)\geq 0$)
\begin{equation}
\label{int-double}
\int _0 ^{\infty} \int _{\R  ^N} x_N |{\widetilde u}|^{\alpha}{\widetilde u}(t,x)\,dxdt< +\infty.
\end{equation}

Next we show the following estimate.

\begin{proposition}[Integral estimate]\label{lem:Estim-T-epsilon} Assume that $A$ satisfies the conditions of Theorem \ref{th:systematic-1}. Let $\alpha = \beta/(N+1)$ and let $u$ be a nonnegative global solution to the nonlinear equation \eqref{eq:1}. Then, for any $T > 0$, there exists a constant $C := C(T) > 0$ such that for any given $\epsilon > 0$ we have
\begin{equation}\label{eq:Estim-T-epsilon}
\int_T^{\infty}\!\!\!\! \int _{\R^{N}_{+}} x_N u^{1+\alpha}(t,x)dxdt \leq C\ep^{\frac{\beta \alpha}{\alpha + 1}} 
\left(\int_{T}^{\infty}\!\!\!\!\int_{\R^{N}_{+}} x_N u^{1+\alpha}(t,x)\right)^{1/(\alpha+1)}. 
\end{equation}
\end{proposition}

\begin{proof}
For given $\epsilon, T > 0$, we take a smooth function $\rho \in C_c^\infty(\R)$ such that 
$$\mathbf 1_{[-1,+1]} \leq \rho \leq \mathbf 1_{[-2,+2]},\qquad \rho(-s) = \rho(s),$$
and then, for $R>0$, we define the truncation functions $\psi_{R}$ and $\zeta_{R}$  by
$$
\psi_R(t) := \rho\left(\frac{t-T}{R^{\beta}}\right) \qquad\text{and}\qquad 
\zeta _R(x) := \rho\left(\ep \frac{| x|}{R}\right).
$$
For the sake of clarity, we divide the remainder of our proof into three steps. In the sequel, we denote by $C$ a positive constant that may change from place to place but that is always independent on  $\ep > 0$ and $R > 0$.

\medskip

\noindent{\bf Step 1.} We multiply equation \eqref{eq:1-RN} by $x_N\zeta _R(x)\psi _R(t)$ and integrate on $(T,+\infty)\times\R^N$, to obtain
\begin{eqnarray}
\int_T^{\infty}\!\!\!\! \int_{\R^{N}} x_N |{\widetilde u}|^{\alpha}{\widetilde u}(t,x)
\zeta_R(x)\psi_R(t) dxdt & = & 
-\int_T^{\infty}\!\!\!\!\int_{\R^{N}}A{\widetilde u}(t,x)x_N\zeta _R(x)\psi_R(t) dxdt \nonumber\\
&& + \int _T^{\infty}\!\!\!\!\int_{\R^{N}}  \partial_t {\widetilde u}(t,x) x_N \zeta _R(x)\psi_R(t) dxdt .\label{equality-1}
\end{eqnarray}
Using the fact that $A$ is self-adjoint we obtain,
$$
\int_T^{\infty}\!\!\!\!\int_{\R^{N}}\!\!\!\! A{\widetilde u}(t,x)x_N\zeta _R(x)\psi_R(t) dxdt =
\int_T^{\infty}\!\!\!\!\int_{\R^{N}}\!\!\!\! (A(x_N\zeta_R))(x){\widetilde u}(t,x)\psi_R(t) dxdt,$$
and for later use we set
\begin{equation}\label{eq:Def-I-1}
I_{1} = I_{1}(R) :=
\int_T^{\infty}\!\!\!\!\int_{\R^{N}}\!\!\!\! (A(x_N\zeta_R))(x){\widetilde u}(t,x)\psi_R(t) dxdt. 
\end{equation}
Also, thanks to the fact that $\psi_{R}$ has a compact support, upon integrating by parts in the time variable, and using the fact that $x_{N}{\widetilde u}(T,x) \geq 0$ on $\R^N$, we have
$$\int _T^{\infty}\!\!\!\!\int_{\R^{N}}  \partial_t {\widetilde u}(t,x) x_N \zeta _R(x)\psi_R(t) dxdt \leq - I_{2} ,$$
where for convenience we have set
$$I_{2} = I_{2}(R) := \int_T^{\infty}\!\!\!\!\int_{\R^{N}} {\widetilde u}(t,x) x_N \zeta _R(x)\psi_R'(t)dxdt.$$
Reporting these into \eqref{equality-1} we obtain
\begin{equation}\label{ineg}
\int_T^{\infty}\!\!\!\! \int_{\R^{N}} x_N |{\widetilde u}|^{\alpha}{\widetilde u}(t,x)
\zeta_R(x)\psi_R(t) dxdt \leq 
-I_{1} - I_{2}.
\end{equation}
Our aim is to prove that as $R \to +\infty$ the above inequality yields \eqref{eq:Estim-T-epsilon}. 

\medskip

\noindent{\bf Step 2.} We begin by proving that as $R \to +\infty$ we have $I_{2}(R) \to 0$. Indeed, observe that since
$$
| \psi_R'(t)| =\left| \frac{1}{R^{\beta}}\rho '\left(\frac{t-T}{R^{\beta}}\right)\right | \leq \frac{C}{R^{\beta}}\mathbf{1}_{(T+R^{\beta},T+2R^{\beta})}(t),
$$
using H\"older's inequality we may write (omitting to write the differential element $dxdt$ and using the shortcut $\widetilde u^{1+\alpha}$ for $|\widetilde u|^\alpha \widetilde u$)
\begin{eqnarray}
|I_2| & \leq & 
\frac{C}{R^{\beta}}\int _{T+R^{\beta}}^{T+2R^{\beta}}\!\!\!\! \int_{|x| \leq 2R/\ep } 
| x_N {\widetilde u}(t,x)| \label{footstep}\\
&\leq &  
\!\!\!\! \frac{C}{R^{\beta}}\left(\int _{T+R^{\beta}}^{T+2R^{\beta}}\!\!\!\!
\int_{|x| \leq 2R/\ep }\!\!\!\! |x_N|\right)^{\frac{\alpha}{\alpha+1}}
\!\!\! \left(\int _{T+R^{\beta}}^{T+2R^{\beta}} 
\!\!\!\!\int_{|x| \leq 2R/\ep }\!\!\!\! |x_N {\widetilde u}^{1+\alpha}(t,x)| \right)^{\frac{1}{\alpha+1}}\nonumber\\
&\leq & 
\frac{C}{R^{\beta}} 
\left( R^{\beta} \left(\frac{2R}{\ep} \right)^{N+1}\right)^{\frac{\alpha}{\alpha+1}}\left(\int_{T+R^{\beta}}^{T+2R^{\beta}}
\!\!\!\!\int_{| x| \leq 2R/\ep }\!\!\!\! |x_N {\widetilde u}^{1+\alpha}(t,x)|\right)^{\frac{1}{\alpha+1}} \nonumber\\
&\leq &
\frac C{\ep^{\frac{(N+1)\alpha}{\alpha+1}}} \left(\int _{T+R^{\beta}}^{T+2R^{\beta}}
\!\!\!\!\int_{|x| \leq 2R/\ep } \!\!\!\! | x_N {\widetilde u}^{1+\alpha }(t,x)|\right)^{\frac{1}{\alpha+1}},\nonumber
\end{eqnarray}
where, in the last step, we have used the fact that $\beta = (N+1)\alpha$. In view of \eqref{int-double}, the last integral above tends to zero as $R\to +\infty$, and so does $I_2$. 

\medskip

\noindent{\bf Step 3.} We now deal with the term $I_1 = I_1(R)$ defined in \eqref{eq:Def-I-1}. Using \eqref{eq:xN-A-zeta-1} of Lemma \ref{lem:A-zeta-R}, upon replacing $R$ by $R/\epsilon$, we have
\begin{equation}\label{hyp-locale-1}
x_N (A(x_N\zeta_R))(x)\geq 0, \quad \text{ for all $x$ such that $|x| \geq \frac{2R}{\ep}$.} 
\end{equation}
Since  ${\widetilde u}(t,x)$  has the same sign as $x_N$, we also have $(A(x_N\zeta_R))(x)\widetilde u(t,x) \geq 0$ for $|x| \geq 2R/\epsilon$, and therefore
\begin{equation}\label{I3}
I_1\geq I_3 := \int_T^{T+2R^{\beta}}\!\!\!\!
\int_{|x| <2R/\ep}\!\!\!\! (A(x_N\zeta_R))(x){\widetilde u}(t,x)\psi_R(t)dxdt.
\end{equation}
Now using \eqref{eq:xN-A-zeta-2}, and replacing $R$ by $R/\epsilon$, we have
\begin{equation}\label{hyp-locale-2}
|(A(x_N\zeta_R))(x)| \leq C | x_N| \frac{\ep^{\beta}}{R^{\beta}},\quad \text{ for all $x$ such that $| x| \leq \frac {2R}\ep $,} 
\end{equation}
so that we may estimate $I_{3}$ from below, that is
\begin{eqnarray*}
I_{1} \geq I_{3} &\geq &
-\int_T^{T+2R^{\beta}}\!\!\!\!
\int_{|x| <2R/\ep}\!\!\!\! \left|(A(x_N\zeta_R))(x)\right| \cdot|{\widetilde u}(t,x)|dxdt \\
& &\quad \geq 
- \; C \frac{\ep^{\beta}}{R^{\beta}}\int _T^{T+2R^{\beta}}\!\!\!\!
\int_{| x|  <2R/\ep} \!\!\!\! |x_N {\widetilde u}(t,x)|dxdt.
\end{eqnarray*}
In the same spirit as in the proof \eqref{footstep}, using H\"older inequality and the fact that $\beta = (N+1)\alpha$, we obtain
\begin{equation}
\label{I3-bis}
-I_{1} \leq - I_{3} \leq  C{\ep^{\frac{\beta \alpha}{\alpha+1}}} 
\left(\int _{T}^{T+2R^{\beta}}\!\!\!\!
\int_{| x| \leq 2 R/\ep } \!\!\!\! |x_N{\widetilde u}^{1+\alpha}(t,x)| dxdt\right)^{\frac{1}{\alpha+1}}.
\end{equation}

\medskip

To conclude, plugging  \eqref{I3-bis} into \eqref{ineg}, we get
\begin{eqnarray*}
0 \leq \int _T^{\infty}\!\!\!\!
\int_{\R^{N}}\!\!\!\! x_N |{\widetilde u}| ^{\alpha} {\widetilde u} (t,x)\zeta _R(x)\psi_R(t) dxdt &\leq& \\
&& \hskip -4cm 
|I_{2}| + C{\ep^{\frac{\beta \alpha}{\alpha+1}}} \left(\int _{T}^{T+2R^{\beta}}
\!\!\!\!\int_{|x| \leq 2 R/\ep } |x_N{\widetilde u}^{1+\alpha}(t,x)| dxdt\right)^{\frac{1}{\alpha+1}}.
\end{eqnarray*}
Letting $R\to +\infty$ yields
$$
\int _T ^{\infty}\!\!\!\!
\int _{\R^{N}}\!\! |x_N{\widetilde u}^{1+\alpha}(t,x)|  dxdt\leq C\ep^{\frac{\beta \alpha}{\alpha+1}} 
\left(\int _{T}^{\infty}\!\!\!\!
\int_{\R ^{N}} \!\!|x_N{\widetilde u}^{1+\alpha}(t,x)| dxdt\right)^{\frac{1}{\alpha+1}},
$$
which is another form of \eqref{eq:Estim-T-epsilon}, since ${\widetilde u}$ is the odd extension of $u$, and the proof of Proposition \ref{lem:Estim-T-epsilon} is complete.
\end{proof}

We are now in the position to complete the proof of Theorem \ref{th:systematic-1}.

\begin{proof}[Proof of Theorem \ref{th:systematic-1}]   If $u\not\equiv 0$ were a nonnegative global solution of \eqref{eq:1}, thanks to \eqref{eq:Estim-T-epsilon} which is valid for any $\ep > 0$ and $T > 0$, letting $\epsilon \to 0$ we would deduce that $u\equiv 0$ on $(T,+\infty)\times \R^{N}$ for any $T > 0$, which is a contradiction. 
\end{proof}

\section{Possible extinction when $\alpha>\beta/(N+1)$}\label{s:extinction}

In this section, we prove that, when $\alpha > p_F - 1 = \beta/(N+1)$, nonnegative solutions starting from small enough initial data are global and go extinct, as stated in Theorem \ref{th:global}. The proof, as the one in  \cite{Gar-Qui-10} or  \cite{Alf-fujita}, relies strongly on the rate of decay of the $L^{\infty}$ norm of the solution of the linear equation $\partial _t v=Av$, which was provided in Lemma \ref{lem:decrease}.

\begin{proof}
[Proof of Theorem \ref{th:global}] We first consider a nonnegative solution $u$ defined on $(0,T)\times \R^N_{+}$ satisfying
$$
\partial_t u = Au + u^{1+\alpha}\quad \text{ in } (0,T)\times \R^N_+,
$$
together with $u(t,x',0) = 0$ on $[0,T] \times \R^{N-1}$, and $u(0,x) = u_0(x)$ on $\R^{N}_+$. We then look after a  supersolution to this problem  in the form $z(t,x) := g(t)v(t,x)$, where $g(t) > 0$ is to be determined appropriately (with $g(0)=1$) and $v(t,x)$ is the solution of $\partial _t v = Av$ with the initial data $v(0,\cdot) = u_0$. Notice in particular that $v\geq 0$ on $(0,+\infty)\times \R^N_+$. A straightforward computation shows that in order to have
$$\partial_{t}z \geq Az + z^{\alpha +1} 
\quad \text{ in } (0,+\infty)\times \R^N_+,$$
it is enough to have
$$
\frac{g'(t)}{g^{1+\alpha}(t)}\geq \|v(t,\cdot)\|_{L^\infty}^{\alpha}.
$$
By Lemma \ref{lem:decrease}, it is therefore enough to have 
$$
\frac{g'(t)}{g^{1+\alpha}(t)} = \frac{C^{\alpha}(m_1(u_0)+\|\widehat {\widetilde{u_0}}\|_{L^1})^{\alpha}}{(1+t)^{\alpha(N+1)/\beta}}, \quad g(0)=1.
$$
Therefore, if one chooses $u_{0}$ such that (recall that here $\alpha(N+1) > \beta$ so that $\epsilon_{*} > 0$ is well defined),
$$
m_1(u_0) + \|\widehat{\widetilde {u_0}}\|_{L^1} < \epsilon_{*} := 
\frac{1}{C}\left(\frac{\alpha(N+1) - \beta}{\alpha\beta}\right)^{1/\alpha}
$$
 then the solution of the above ODE Cauchy problem 
$$
g(t) = 
\left(
1 - \frac{\beta \alpha C^{\alpha}(m_1(u_0)+\|\widehat {{\widetilde u}_0}\|_{L^1})^{\alpha}}
{\alpha(N+1)-\beta}
\left(1 - (1+t)^{1 - \frac{\alpha(N+1)}{\beta}} \right) \right)^{-1/\alpha},
$$
exists for all $t\geq 0$, is increasing and bounded as $t \to +\infty$. It therefore follows from the comparison principle that $0\leq u(t,x)\leq z(t,x) = g(t)v(t,x)\leq \|g \Vert_\infty v(t,x)$  in the half-space and thus $| {\widetilde u}(t,x)| \leq \|g\|_\infty | {\widetilde v}(t,x)|$ in the whole space. Hence the solution $u(t,x)$ of \eqref{eq:1} is global in time and, in view of Lemma \ref{lem:decrease}, satisfies estimate \eqref{extinction}. This concludes the proof of Theorem \ref{th:global}.
\end{proof}

\medskip 

\noindent{\bf Acknowledgements.}   M. Alfaro is supported by the 
ANR \textsc{i-site muse}, project \textsc{michel} 170544IA (n$^{\circ}$ ANR \textsc{idex}-0006).


\bibliographystyle{siam}  

\bibliography{biblio}

\end{document}